\documentclass[10pt]{amsart}
\usepackage{amsmath}             
\usepackage{amssymb}             
\usepackage{latexsym}            

\usepackage{ae}
\usepackage[dvips]{graphicx}

\numberwithin{figure}{section}

\newtheorem{theorem}{Theorem}[section]
\newtheorem{lemma}[theorem]{Lemma}

\theoremstyle{definition}
\newtheorem{definition}[theorem]{Definition}

\newtheorem{remark}[theorem]{Remark}
\numberwithin{equation}{section}
\newcommand{\abs}[1]{\lvert#1\rvert}
\newcommand{\R}{\mathbb{R}}

\newcommand{\dist}{{\rm dist}}

\newcommand{\dom}{\Omega}
\newcommand{\eps}{\varepsilon}

\newcommand{\diver}{\operatorname{div}}
\begin{document}

\title[A new proof for the equivalence]{A new proof for the equivalence of weak and viscosity solutions
for the $p$-Laplace equation}

\author{Vesa Julin}
\address{Dipartimento di Matematica e Applicazioni ''R. Cacciopoli'', Universita degli Studi di Napoli ''Federico II'', Napoli, Italy}
\email{vesa.julin@jyu.fi}

\author{Petri Juutinen}

\address{Department of Mathematics and Statistics,
P.O.Box 35, FIN-40014 University of Jyv\"askyl\"a, Finland}
\email{petri.juutinen@jyu.fi}

\thanks{}

\keywords{$p$-Laplace equation, weak solutions, viscosity solutions}
\subjclass[2000]{35J92, 35D40, 35D30}
\date{\today}

\begin{abstract} In this paper, we give a new proof for the fact that the distributional weak solutions and
the viscosity solutions of the $p$-Laplace equation $-\diver(\abs{Du}^{p-2}Du)=0$ coincide. Our proof is 
more direct and transparent than the original one by Juutinen, Lindqvist and Manfredi \cite{jlm}, which relied on 
the full uniqueness machinery of the theory of viscosity solutions. We establish a similar result also for the  
solutions of the non-homogeneous version of the $p$-Laplace equation.
\end{abstract}
\maketitle

\section{Introduction}

The $p$-Laplace equation
\begin{equation}\label{eq:plap}
-\Delta_p u(x):= -\diver(\abs{Du(x)}^{p-2}Du(x))=0,
\end{equation}
where $u$ is a scalar function and $1<p<\infty$,
is the Euler-Lagrange equation for the energy functional
\[
 \int_\Omega \abs{Du(x)}^p\, dx.
\]
It is a generalization of the classical Laplace equation and can be viewed as a prototype of a quasilinear
elliptic equation exhibiting $p$-growth. The $p$-Laplace operator appears in numerous physical applications
(see e.g. \cite{jlm} and references therein) and is fundamental in the nonlinear potential theory, see \cite{hkm}.

The equation \eqref{eq:plap} is degenerate for $p> 2$ and singular for $1<p<2$, and to assure the solvability
of the Dirichlet boundary value problem one usually resorts to the distributional weak solutions, whose
definition is based on integration by parts. However, it is often desirable to have a pointwise interpretation for 
identities involving the second derivatives of a (super)solution, even though these derivatives need not really exist.
For example, the identity
\[
(r-q)\abs{Du}^{2-p} \Delta_p u = (r-p)\abs{Du}^{2-q} \Delta_q u + (p-q)\abs{Du}^{2-r} \Delta_r u,
\]
where $1<q<p<r<\infty$, suggests that a function which is a (super)solution to both $-\Delta_q u=0$
and $-\Delta_r u=0$ is a (super)solution to  the equation $-\Delta_p u=0$ as well. This claim can easily be made rigourous 
in the framework of viscosity solutions, but to conclude that it holds for
distributional weak solutions one needs to know that the weak and viscosity solutions 
coincide. 
 
The equivalence of the distributional weak solutions and viscosity solutions for the $p$-Laplace equation \eqref{eq:plap}
was established by Juutinen, Lindqvist and Manfredi in \cite{jlm}. The proof relied on the full uniqueness machinery of 
the theory of viscosity solutions, including the maximum principle for semicontinuous functions, and used the
structural properties of the $p$-Laplace operator in an essential way. In this paper, we give a new, more direct, proof for 
the equivalence that applies to various other equations as well. Moreover, just as in \cite{jlm}, we actually show that the
equivalence holds also for supersolutions, that is, weak supersolutions and viscosity supersolutions coincide.  

One of the implications, that weak solutions are viscosity solutions, is easy and appeared already in \cite{ju}.
We have nothing new to say about it, so let us focus on the converse.
In \cite{jlm}, the converse was established by showing, roughly speaking, that the Dirichlet boundary value problem
has a unique viscosity solution. Since the boundary value problem also has a weak solution and this
weak solution is already known to be a viscosity solution, it follows that the viscosity solution must be a
weak solution. Our new proof is completely different as we take as our starting point the identity
\begin{equation}\label{integ_parts}
\int_\dom (-\Delta_p u)\psi\, dx = \int_\dom \abs{Du}^{p-2} Du\cdot D\psi\, dx,
\end{equation}
which holds if $u$ and $\psi$ are sufficiently smooth and $\psi$ is compactly supported. To show that a viscosity supersolution 
$u$ is a weak supersolution, we perform a sequence of approximations that enable us to make sense of
(and have a right sign for) $-\Delta_p u$ at sufficiently many points. In view of \eqref{integ_parts},
this then shows that  
$\int_\dom \abs{Du}^{p-2} Du\cdot D\psi\, dx \ge 0$ for all non-negative $\psi$, as desired.

In the degenerate case $p\ge 2$, our ideas are fairly easy to carry out and this is done in Section \ref{sec:dege} below.
However, the singular case $1<p<2$ is more difficult, since it is not even clear what is the meaning of  $-\Delta_p u$ at points
where the gradient vanishes and hence the interpretation of \eqref{integ_parts} requires
some thought. Nevertheless, it turns out that by suitably modifying the argument used for $p\ge 2$, we obtain
the results in the singular case as well. The main new ingredient is that we use an approximation procedure
that depends on the exponent $p$ and effectively cancels out the singularity of the equation, see Section \ref{sec:singular}.

Our method extends to the non-homogeneous equation 
\[
-\Delta_p u(x)=f(x), 
\]
where $f$ is continuous, without much difficulties. Moreover, it also applies,
for example, to the minimal surface equation
\[
-\diver(\frac{Du}{\sqrt{1+\abs{Du}^2}})=f(x).
\]
On the other hand, it seems that in order to cover general state dependent equations of the form 
$-\diver A(x,Du)=0$ some new arguments are needed.

In addition to the interpretation of pointwise identies (see \cite{lms} for more), 
the equivalence of the distributional weak solutions and viscosity 
solutions has also other applications, some of which are a bit surprising. The equivalence has turned out to be a very useful tool 
in certain removability questions, see \cite{jl} and \cite{p}, as well as in the analysis of various approximations of the 
$p$-Laplace equation \cite{in}, \cite{mpr}, \cite{ps}. 

\tableofcontents

\section{Notions of solutions}
\label{sec:notions}

In this section, we recall the notions of weak solutions,
$p$-super\-har\-monic functions and viscosity solutions. 

\begin{definition}
  \label{def:div-weak} Let $f\colon\Omega\to\R$ be continuous. 
  A function $u \in W^{1,p}_{\text{loc}}(\Omega)$ is a \emph{weak supersolution}
  to $-\Delta_p v(x)=f(x)$ in $\Omega$ if
  \begin{equation}
    \label{eq:div-weak}
    \begin{split}
      \int_\Omega \abs{D u(x)}^{p-2}Du(x)\cdot D\psi(x)\,dx\geq \int_\Omega \psi(x)f(x)\,dx
    \end{split}
  \end{equation}
  for every nonnegative test function $\psi \in C^{\infty}_0(\Omega)$.

  A function $u \in W^{1,p}_{\text{loc}}(\Omega)$ is a \emph{weak subsolution} to $-\Delta_p v(x)=f(x)$,
 if $-u$ is a weak supersolution, and 
  a \emph{weak solution}, if it is both a super-- and a subsolution, which
 is equivalent to saying that we have an equality in \eqref{eq:div-weak} for all $\psi \in
  C^{\infty}_0(\Omega)$.
\end{definition}

In the case $f\equiv 0$, the class of weak supersolutions is not adequate for the purposes of nonlinear potential theory, and
thus a larger class of ``supersolutions'' is needed. In what follows, we call a function $p$-harmonic, if it is a
continuous weak solution to $-\Delta_p v(x)=0$.

\begin{definition}
  \label{def:div-harmonic}
  A function $u: \Omega\to  (-\infty,\infty]$ is \emph{$p$-superharmonic}, if
  \begin{enumerate}
  \item $u$ is lower semicontinuous
  \item $u$ is not identically $+\infty$ and
  \item \label{itm:superharm-comp} the comparison principle holds: if
    $h$ is $p$-harmonic in
    $D\subset\subset\Omega$, continuous in $\overline D$, and
    \[
    u \geq  h \quad \text{on}\quad \partial D,
    \]
    then
    \[
    u\geq   h  \quad \text{in}\quad D.
    \]
  \end{enumerate}
A function $u: \Omega\to  [-\infty,\infty)$ is \emph{$p$-subharmonic}, if $-u$ is
$p$-super\-har\-monic.
\end{definition}

The exact relationship between weak supersolutions and $p$-superharmonic functions is one of the main
concerns in the nonlinear potential theory. We refer to \cite{l}, \cite{hkm} for the proof of the following facts:

\begin{enumerate}
 \item Every weak supersolution to $-\Delta_p v(x)=0$
has a lower semicontinuous representative which is $p$-superharmonic.
\item If $u$ is $p$-superharmonic, then the truncations $\min(u,k)$, $k\in\R$, are weak
  supersolutions to $-\Delta_p v(x)=0$. In particular, a locally bounded $p$-superharmonic function is a weak
  supersolution.
\end{enumerate}

Finally, let us discuss the definition of viscosity solutions for the equation $-\Delta_p v(x)=f(x)$. 
For $p\ge 2$, the equation is pointwise well-defined, and the standard definition, see e.g \cite{cil}, can be used as it is.
However, if $1<p<2$, then the equation is singular and extra caution is needed at the points where
the gradient vanishes.  

\begin{definition}
  \label{def:div-viscosity}
  A function $u:\Omega\to (-\infty,\infty]$ is a viscosity
  supersolution to $-\Delta_p v(x)=f(x)$ in $\Omega$, if
  \begin{enumerate}
  \item $u$ is lower semicontinuous.
  \item $u$ is not identically $+\infty$, and
  \item \label{itm:visco-def-3} If $\psi\in C^2(\Omega)$ is such that $u(x_0) = \psi(x_0)$, $u(x) \geq \psi(x)$
for $x \neq x_0$, and $D\psi(x)\ne 0$ for $x \neq x_0$, 
   it holds that
    \[
    \begin{split}
      \lim_{r\to 0}\sup_{x\in B_r(x_0)} (-\Delta_{p} \psi(x))\geq f(x_0).
    \end{split}
    \]
  \end{enumerate}

  A function $u$ is a viscosity subsolution to
  $-\Delta_p v(x)=f(x)$, if $-u$ is a viscosity supersolution, and 
 a viscosity solution, if it is both a
  viscosity super- and subsolution.
\end{definition}

There are several equivalent ways to formulate the definition above, see e.g. \cite{is}. 
If $f\equiv 0$, then one can even completely ignore the test functions whose gradient
vanishes at the contact point, see \cite{jlm}. For the non-homogeneous equation
this is obviously not a good idea.

\section{The non-singular case $p\ge 2$}\label{sec:dege}

We begin by proving our main result in the simplest setting, that is, $p\ge 2$ and $f\equiv 0$. 

\begin{theorem}
\label{main1}
If $u$ is a viscosity supersolution to \eqref{eq:plap}, then it is
$p$-superharmonic.
\end{theorem}

\begin{proof} 
By replacing $u$ by $\min(u,k)$ for $k\in\R$, we may assume that $u$ is locally bounded. Indeed, 
it is easy to check that also $\min(u,k)$ is a viscosity supersolution to \eqref{eq:plap}, and 
if we can show that it is $p$-superharmonic, the $p$-superharmonicity of $u$ follows easily from Definition \ref{def:div-harmonic},
see e.g. \cite{hkm}. 

We use the standard inf-convolution
\begin{equation}
\label{inf-conv}
u_\eps(x)= \inf_{y\in\dom} \left( u(y)+\frac1{2\eps}\abs{x-y}^2 \right),
\end{equation}
that is, the convolution (\ref{inf-conv-G}) with $q=2$. According to Lemma \ref{infprop}, 
$(u_\eps)$ is an increasing sequence of semiconcave viscosity supersolutions to \eqref{eq:plap} in $\Omega_{r(\eps)}$ 
which converge pointwise to $u$. In particular, the function 
\[
\varphi(x) = u_{\eps}(x) - C|x|^2
\]
(where we may take $C = \frac{1}{2 \eps}$) is concave in $\Omega_{r(\eps)}$. 

By Aleksandrov's theorem, $u_\eps$ is twice differentiable a.e. and we have\footnote{It follows 
from the definitions that the pair $(Du_\eps(x), D^2u_\eps(x))$ belongs 
to the second order ``jets'' $J^{2,+}u(x)$ and $J^{2,-}u(x)$ at the points of twice differentiability.}    
\begin{equation}\label{claim1}
-\Delta_p u_\eps = -\abs{Du_\eps}^{p-2}\left(\Delta u_\eps +(p-2) D^2 u_\eps 
\frac{Du_\eps}{\abs{Du_\eps}}\cdot\frac{Du_\eps}{\abs{Du_\eps}}\right) \ge 0
\end{equation}
a.e. in $\Omega_{r(\eps)}$. Here $D^2 u_\eps(x)$ is the Hessian matrix (in the sense of 
Aleksandrov) of $u_\eps$ at $x$, and $\Delta u_\eps(x)$ denotes the trace of this matrix.

Owing to \eqref{claim1} and \cite[Theorem 2.3]{l}, it is enough to prove that
\begin{equation}\label{enough}
 \int_\dom \abs{Du_\eps}^{p-2} Du_\eps\cdot D\psi\, dx  \geq \int_\dom (-\Delta_p u_\eps)\psi\, dx
\end{equation}
for any non-negative $\psi\in C^\infty_0(\dom)$. Let us fix such a function $\psi$. Notice that by Lemma \ref{infprop}, 
$ \text{supp} \, \psi \subset \Omega_{r(\eps)}$ when $\eps$ is small. Let $\varphi_j$ be a sequence of smooth concave functions
converging to $\varphi$, obtained via standard mollification, and let $u_{\eps,j}=\varphi_j+\frac1{2\eps}\abs{x}^2$. 
Integration by parts gives
\begin{equation}\label{div_thm}
\int_\dom \abs{Du_{\eps,j}}^{p-2} Du_{\eps,j}\cdot D\psi\, dx = \int_\dom (-\Delta_p u_{\eps,j})\psi\, dx .
\end{equation}
Since $u_\eps$ is locally Lipschitz continuous, we clearly have
\[
\lim_{j\to\infty} \int_\dom \abs{Du_{\eps,j}}^{p-2} Du_{\eps,j}\cdot D\psi\, dx 
= \int_\dom \abs{Du_\eps}^{p-2} Du_\eps\cdot D\psi\, dx.
\]
On the other hand, since $D^2 u_{\eps,j}\le \frac1\eps I$ and $Du_{\eps,j}$ is locally bounded, we have
\[
-\Delta_p u_{\eps,j} \ge  -\frac{C^{p-2}(n+p-2)}\eps
\]
in the support of $\psi$. Thus, by Fatou's lemma,
\begin{equation}
\label{fatou}
\liminf_{j\to\infty} \int_\dom (-\Delta_p u_{\eps,j})\psi\, dx 
\ge \int_\dom \liminf_{j\to\infty} (-\Delta_p u_{\eps,j})\psi\, dx.
\end{equation}
Finally, it is shown in \cite[p. 242]{eg}, that $D^2 \varphi_j(x)\to D^2\varphi(x)$ for a.e. $x$, and thus
\[
\liminf_{j\to\infty} (-\Delta_p u_{\eps,j}(x)) = -\Delta_p u_{\eps}(x)
\] 
for a.e. $x$. Putting everything together, we have
\[
\begin{split}
\int_\dom \abs{Du_\eps}^{p-2} Du_\eps\cdot D\psi\, dx 
=& \lim_{j\to\infty} \int_\dom \abs{Du_{\eps,j}}^{p-2} Du_{\eps,j}\cdot D\psi\, dx \\
=& \lim_{j\to\infty} \int_\dom (-\Delta_p u_{\eps,j})\psi\, dx\\
\ge& \int_\dom \liminf_{j\to\infty} (-\Delta_p u_{\eps,j})\psi\, dx\\
=& \int_\dom (-\Delta_p u_\eps)\psi\, dx,
\end{split}
\]
as desired.
\end{proof}

The argument used to prove Theorem \ref{main1} extends without much difficulties for the non-homogeneous equation
$-\Delta_p u(x)=f(x)$, where $f$ is any continuous function.    

\begin{theorem}\label{fx1}
Let $u$ be a locally bounded viscosity supersolution to $-\Delta_p u(x)=f(x)$ in $\Omega$. Then $u$
is also a weak supersolution to the same equation. 
\end{theorem}

\begin{proof}
Let $u_\eps$ be the standard inf-convolution of $u$ as above. Then $u_\eps$ is a semi-concave viscosity supersolution
to
\[
-\Delta_ p v(x)=f_\eps(x) \qquad \text{in} \, \, \Omega_{r(\eps)},
\]  
where $f_\eps(x)= \inf\limits_{y\in B_{r(\eps)}(x)} f(y)$. Arguing as in the proof of Theorem \ref{main1},
we obtain
\[
\int_\dom \abs{Du_\eps}^{p-2} Du_\eps\cdot D\psi\, dx  \ge \int_\dom (-\Delta_p u_\eps)\psi\, dx 
\ge \int_\dom f_\eps \psi\, dx.
\]
In view of standard Caccioppoli estimates, the claim now follows by letting $\eps\to 0$.
\end{proof}

\begin{remark} One of the crucial points in the proof of Theorem \ref{main1} is the convergence
of $D^2 \varphi_j(x)$ to $D^2\varphi(x)$ for a.e. $x\in\Omega$. In fact, this holds if
\begin{enumerate}
 \item $\varphi$ is differentiable at $x$ and $x$ is a Lebesgue point for $D\varphi$,
\item $x$ is a Lebesgue point for the absolutely continuous
part of the measure $[D^2\varphi]$, and
\item $x$ is not a density point for the singular part of the measure $[D^2\varphi]$.
\end{enumerate}
See \cite[Section 6.3]{eg} for details.
\end{remark}

\section{The singular case $1< p <2$.}\label{sec:singular}

In this section, we show that the argument of the previous section can be extended to cover the case $1<p<2$ as well.
However, in this range, the $p$-Laplace operator is singular at the points where the gradient vanishes, and this
fact causes several difficulties that we did not encounter before.

The first problem is that it is not obvious what is the meaning of the expression 
\[
\Delta_p u=\abs{Du}^{p-2} \left(\Delta u+(p-2)D^2u \tfrac{Du}{\abs{Du}}\cdot \tfrac{Du}{\abs{Du}}\right)
\]
when $Du=0$. Because of this, we cannot just integrate by parts and conclude that
\[
 \int_\dom (-\Delta_p u)\psi\, dx = \int_\dom \abs{Du}^{p-2} Du\cdot D\psi\, dx
\]
holds for smooth functions, cf. \eqref{div_thm}. To circumvent this, we regularize the equation 
and use the identity
\[
\int_\dom (-\diver((\abs{Du}^2+\delta)^{\frac{p-2}2} Du))\psi\, dx = \int_\dom (\abs{Du}^2+\delta)^{\frac{p-2}2} Du\cdot D\psi\, dx,
\]
and eventually try to pass to the limit as $\delta\to 0$.

Second, even though the standard inf-convolution \eqref{inf-conv} produces 
semi-concave supersolutions, this doesn't imply directly that the expressions
\[
-\Delta_p u_\eps(x)
\]
have an integrable lower bound. Therefore we cannot justify the use 
of Fatou's Lemma as we did in \eqref{fatou}. To overcome this problem, we 
use a slightly different inf-convolution which will cancel out the singularity of the operator. 

Let us now make all this precise. We consider the equation
\begin{equation}
 \label{poisson}
-\Delta_p u= f
\end{equation}
in a domain $\Omega\subset\R^n$, and assume that $f\in C(\Omega)$ and $1<p<2$.

\begin{theorem}
\label{Main-thm-p<2}
If $u$ is a bounded viscosity supersolution to (\ref{poisson}) in $\Omega$, then it is a weak supersolution to (\ref{poisson}) 
in $\Omega$.
\end{theorem}

The convolution we will use is
\begin{equation}
 \label{inf-conv-2}
u_{\eps}(x) = \inf_{y \in \Omega} \left( \, u(y) + \frac{|x-y|^q}{q \, \eps^{q-1}} \, \right)
\end{equation}
where $q > \frac{p}{p-1}$. Notice that $\frac{p}{p-1}$ is the dual exponent of $p$, and as $1<p<2$, we have $q> 2$. 
By Lemma \ref{infprop}, $u_{\eps}$ is semiconcave and therefore it is twice differentiable almost everywhere. The main advantage in having $q> 2$ is that this will force the Hessian $D^2u_{\eps}(x)$ to be negative semidefinite whenever the gradient $D u_{\eps}(x)$ vanishes. 

We make two observations (Lemmas \ref{lemma_p<2} and \ref{p<2lemma2} below)
about the inf-convolution \eqref{inf-conv-2} before entering the proof of Theorem \ref{Main-thm-p<2}.
The proofs for these observations are given in Appendix \ref{sec:infconv}, where other basic results 
regarding inf-convolutions are also discussed.

First we need to introduce some notation.

\begin{definition} For a bounded, lower semicontinuous function $u\colon\Omega\to\R$ and $x\in\Omega$, we define the set
$Y_{\eps}(x)$ by saying that $y \in Y_{\eps}(x)$ if  
\[
u_{\eps}(x) = u(y) + \frac{1}{q \, \eps^{q-1}}|x-y|^q .
\]
\end{definition}
As $u$ is lower semicontinuous and bounded, the set $Y_{\eps}(x)$ is nonempty and closed for every 
$x \in \Omega_{r(\eps)}$.

\begin{lemma}
\label{lemma_p<2}
Suppose that $u\colon \Omega \to \R$ is bounded and lower semicontinuous. Let $u_{\eps}$ be the inf-concolution of $u$ as in 
\eqref{inf-conv-2} and let $Y_{\eps}(x)$ be the set defined as above. Then the following are true:
\begin{itemize}

 \item[(i)] The function  $x \mapsto \max\limits_{y \in Y_{\eps}(x)} |y-x|$ is upper semicontinuous.

  \item[(ii)] If the gradient $D u_{\eps}(x)$ exists, then for every $y \in Y_{\eps}(x)$ it holds 
\[
\left(\frac{|x -y|}{\eps}\right)^{q-1}  \leq |D u_{\eps}(x)| .
\]
In particular, if $D u_{\eps}(x)= 0$, then $u_{\eps}(x) = u(x)$.
\end{itemize}

\end{lemma}

\begin{lemma}
\label{p<2lemma2}
Suppose that $u$ is a bounded viscosity supersolution to (\ref{poisson}) in $\Omega$. If there is $\hat{x} \in \Omega_{r(\eps)}$ such that $u_{\eps}$ is differentiable at $\hat{x}$ and 
\[
 D u_{\eps}(\hat{x}) = 0,
\]
then $f(\hat{x}) \leq 0$. 
\end{lemma}

\begin{proof}[Proof of Theorem \ref{Main-thm-p<2}] 
The outline of the proof is the same as in the case $p \geq 2$. Let $u_\eps$ be the inf-convolution of $u$,
defined using \eqref{inf-conv-2}. Then $u_\eps$ is a semiconcave viscosity supersolution to 
\[
 -\Delta_p u_{\eps} = f_{\eps}
\]
in $\Omega_{r(\eps)}$. We will show that for any $\eps > 0$, $u_{\eps}$ is also a weak supersolution to the 
same equation.

By Aleksandrov's theorem, we have
\begin{equation}
\label{claim2}
\begin{split}
-\Delta_p u_\eps &= -\diver \left(\abs{Du_\eps}^{p-2} Du_\eps \right) \\
			&= -\abs{Du_\eps}^{p-2}\left(\Delta u_\eps +(p-2) D^2 u_\eps 
\frac{Du_\eps}{\abs{Du_\eps}}\cdot\frac{Du_\eps}{\abs{Du_\eps}}\right) \\
 &\ge f_{\eps}
\end{split}
\end{equation}
a.e. in $\Omega_{r(\eps)} \backslash \{ Du_\eps = 0\}$.
Since $u_{\eps}$ is semiconcave, we can combine a mollification argument with Fatou's Lemma as in \eqref{fatou} 
to obtain
\begin{equation}
\label{regular}
  \int_\dom (\abs{Du_\eps}^2 + \delta )^{\frac{p-2}{ 2}} Du_\eps \cdot D\psi\, dx \geq \int_\dom -\diver \left((\abs{Du_\eps}^2 + \delta )^{\frac{p-2}{ 2}} Du_\eps \right) \,  \psi\, dx
\end{equation}
for $\psi \in C^\infty_0(\dom)$, $ \psi \geq 0$ and for every $\delta > 0$. The goal is to let $\delta \to 0$ in \eqref{regular} and conclude that 
\[
 \int_\dom \abs{Du_\eps}^{p-2} Du_\eps \cdot D\psi\, dx \geq \int_\dom f_{\eps} \,  \psi\, dx.
\]

The convergence of the left-hand side in (\ref{regular}) is fine, but the right-hand side needs an additional argument. 
We will show that 
\begin{equation}
\label{fatou2}
\begin{split}
 & -\diver((\abs{Du_\eps}^2 + \delta )^{\frac{p-2}{ 2}} Du_\eps) \\
 &= -(\abs{Du_\eps}^2 + \delta )^{\frac{p-2}{ 2}}\left(\Delta u_\eps +\frac{(p-2)}{\abs{Du_\eps}^2 + \delta} D^2 u_\eps 
Du_\eps \cdot Du_\eps \right) \\
&\geq  - C
\end{split}
\end{equation}
a.e. in $\Omega_{r(\eps)}$, where $C$ is a constant independent of $\delta$, and use Fatou's Lemma again.  

To show \eqref{fatou2}, consider a point $\hat{x}$ where both $Du_\eps(\hat{x})$ and $D^2 u_\eps(\hat{x})$ exist. 
By Lemma \ref{lemma_p<2} (ii), we have $|y-\hat{x}|\leq |Du_\eps(\hat{x})|^{\frac{1}{q-1}} \eps$ for every $y \in Y_{\eps}(\hat{x})$. 
Moreover, by the upper semicontinuity result in Lemma \ref{lemma_p<2} (i), we know that for every 
$n$ there is a small radius $\rho_n$ such that for all $x \in B_{\rho_n}(\hat{x})$ and for all $y \in Y_{\eps}(x)$ it holds 
\[
 |y-x|\leq |Du_\eps(\hat{x})|^{\frac{1}{q-1}} \eps + \frac{1}{n} =: r_n.
\]
This implies that for every $x \in B_{\rho_n}(\hat{x})$ we have 
\[
u_{\eps}(x) = \inf_{y \in B_{ r_n}(x)} \left( \, u(y) + \frac{|x-y|^q}{q \, \eps^{q-1}} \, \right).
\]
For every $y \in B_{ r_n}(x)$, the function $\varphi_y(x) = u(y) + \frac{|x-y|^q}{q \, \eps^{q-1}}$ is smooth with 
\[
D^2\varphi_y(x) \leq \frac{q-1}{\eps^{q-1}} \,  r_n^{q-2} I.
\]
Since $u_\eps$ is the infimum of $\varphi_y$'s over $y \in B_{ r_n}(x)$, we conclude, as in the proof of Lemma~\ref{semiconcave}, 
that $u_\eps$ is semiconcave and
\[
D^2 u_\eps(x) \leq \frac{q-1}{\eps^{q-1}} \,  r_n^{q-2} I
\]
a.e. in $B_{\rho_n}(\hat{x})$. Letting $n \to \infty$ yields the estimate
\begin{equation}
\label{bound}
 D^2 u_\eps(\hat{x}) \leq \frac{q-1}{\eps} | Du_\eps(\hat{x})|^{\frac{q-2}{q-1}}I.
\end{equation}

The previous estimate proves \eqref{fatou2}. Indeed, by \eqref{bound} we have   
\begin{equation}
 \label{criticalset}
D^2 u_\eps(x) \leq 0 \qquad \text{if} \quad Du_\eps(x)= 0.
\end{equation}
If $Du_\eps(x) \neq 0$, we have 
\[
\begin{split}
 &-(\abs{Du_\eps}^2 + \delta )^{\frac{p-2}{ 2}}\left(\Delta u_\eps +\frac{(p-2)}{\abs{Du_\eps}^2 + \delta} D^2 u_\eps Du_\eps \cdot Du_\eps \right) \\
&\geq -\frac{(n+p-2)(q-1)}{\eps} \abs{Du_\eps}^{p-2 + \frac{q-2}{q-1}} \\
&\geq -C,
\end{split}
\]
where the last inequality follows from $q >\frac{p}{p-1}$ and the Lipschitz continuity of $u_\eps$. 
We may thus use Fatou's lemma to conclude that
\[
\begin{split}
&\liminf_{\delta \to 0} \int_\dom -\diver \left((\abs{Du_\eps}^2 + \delta )^{\frac{p-2}{ 2}} Du_\eps \right) \,  \psi\, dx \\
&\geq \int_{\dom \backslash \{ Du_\eps = 0\} }  \liminf_{\delta \to 0} \left( -\diver \left((\abs{Du_\eps}^2 + \delta )^{\frac{p-2}{ 2}} Du_\eps \right) \,  \psi\, \right) dx \\
&= \int_{\dom \backslash \{ Du_\eps = 0\} }  -\diver \left(\abs{Du_\eps}^{p-2} Du_\eps \right) \,  \psi\,  dx \\
&= \int_{\dom \backslash \{ Du_\eps = 0\} } -\Delta_p u_\eps  \,  \psi\,  dx \\
&\geq \int_{\dom \backslash \{ Du_\eps = 0\} } f_{\eps} \,  \psi\,  dx,
\end{split}
\]
where the first inequality follows from \eqref{criticalset} and Fatou's lemma, and the last inequality follows from \eqref{claim2}. Let $\delta \to 0$ in (\ref{regular}) to obtain
\[
\int_\dom \abs{Du_\eps}^{p-2} Du_\eps \cdot D\psi\, dx \geq \int_{\dom \backslash \{ Du_\eps = 0\} } f_{\eps} \,  \psi\,  dx.
\]
Finally, we use Lemma \ref{p<2lemma2} to conclude that $f_{\eps} \leq f \leq 0$ on the set $ \{ Du_\eps = 0\}$. Therefore 
\[
 \int_{\dom \backslash \{ Du_\eps = 0\} } f_{\eps} \,  \psi\,  dx  \geq \int_\dom  f_{\eps} \,  \psi\,  dx 
\]
and we are done.
\end{proof}

If $f\equiv 0$, then we have a more elegant statement that was already established for $p\ge 2$ in Theorem \ref{main1}.

\begin{theorem}
If $u$ is a viscosity supersolution to $-\Delta_p u=0$, then it is
$p$-superharmonic.
\end{theorem}

\appendix

\section{Infimal convolutions}
\label{sec:infconv}

In this paper, we use infimal convolutions of general type
\begin{equation}
\label{inf-conv-G}
u_{\eps}(x) = \inf_{y \in \Omega} \left( \, u(y) + \frac{|x-y|^q}{q \, \eps^{q-1}} \, \right),
\end{equation}
where $q \geq 2$. The next lemma contains some basic facts about these operators.
\begin{lemma}
\label{infprop}
Suppose that $u:\Omega \to \R$ is bounded and lower semicontinuous.
\begin{itemize}
\item[(i)] There exists $r(\eps)>0$ such that 
\[
u_{\eps}(x) = \inf_{y \in B_{r(\eps)}(x) \cap\Omega} \left( \, u(y) + \frac{|x-y|^q}{q \, \eps^{q-1}} \, \right).
\]
Moreover, $r(\eps) \to 0$ as $\eps \to 0$.
\item[(ii)] The sequence $(u_{\eps})$ is increasing and $u_{\eps} \to u$ pointwise in $\Omega$.
\item[(iii)] If $u$ is a viscosity supersolution to $-\Delta_p u(x) \geq f(x)$, for $1 < p < \infty$, then $u_{\eps}$ is a viscosity supersolution to 
\[
-\Delta_p u_{\eps}(x) \geq f_{\eps}(x) \qquad  \text{in} \quad \Omega_{r(\eps)},
\] 
where $ f_{\eps}(x) = \inf\limits_{y \in B_{r(\eps)}(x)} f(y)$ and 
$\Omega_{r(\eps)} = \{ x \in \Omega \colon \dist(x, \partial \Omega) > r(\eps) \}$.
\end{itemize}
\end{lemma}

\begin{proof}
These results can be found in the literature, but we prove them here for the readers convenience. For (i) denote $M = \sup_{\Omega}u$ and $ m = \inf_{\Omega} u$ and choose $r = r(\eps)$ such that 
\[
\frac{1}{q}  \frac{r^q}{\eps^{q-1}}  = M -m .
\]
Then $r \to 0$ as $\eps \to 0$ and for all $y \in \Omega \setminus \bar{B}_r(x)$ it holds
\[
u(y)  + \frac{|x-y|^q}{q \, \eps^{q-1}} > m + \frac{r^q}{q \, \eps^{q-1}} \geq  M  \geq u(x),
\]
which proves the claim.

Part (ii) follows directly from the definition of $ u_{\eps}$ and from (i).

To prove (iii), we notice that in view of (i), for every $x \in \Omega_{r(\eps)}$ we have
\[
u_{\eps}(x) = \inf_{y \in B_{r(\eps)}(x)} \left( \, u(y) + \frac{|x-y|^q}{q \, \eps^{q-1}} \, \right) = \inf_{z \in B_{r(\eps)}(0)} \left(  u(z -x) + \frac{|z|^q}{q \, \eps^{q-1}} \, \right).
\]
It is easy to see that for every $z \in  B_{r(\eps)}(0)$ the function $\varphi_z (x) = u(z -x) + \frac{|z|^q}{q \, \eps^{q-1}}$ is a 
viscosity supersolution to $-\Delta_p \varphi_z \geq f_{\eps}$ in $\Omega_{r(\eps)}$. Since $u_{\eps}$ is an infimum over such functions, $u_{\eps}$ itself is a viscosity supersolution to $-\Delta_p u_{\eps} \geq f_{\eps}$ in $\Omega_{r(\eps)}$.
\end{proof}

\begin{lemma}
\label{semiconcave}
Suppose that $u\colon\Omega \to \R$ is bounded and lower semicontinuous. Then $u_{\eps}$ defined by (\ref{inf-conv-G}) 
is semiconcave in $\Omega_{r(\eps)}$, that is, there is a constant $C$ such that the function
\begin{equation}
\label{semiconcave1}
x \mapsto u_{\eps}(x) -  C |x|^2 
\end{equation}
is concave. The constant $C$ depends on $q$, $\eps$ and the oscillation $\sup_{\Omega}u- \inf_{\Omega}u$.
\end{lemma}

\begin{proof}
Fix $x \in \Omega_{r(\eps)}$. For $y \in \Omega \cap  B_{r(\eps)}(x)$, where $r(\eps)$ is the radius appearing in Lemma  
\ref{inf-conv-G} (i),  denote $\varphi_y(x) = u(y) + \frac{1}{q \, \eps^{q-1}}|x-y|^q$. Since $q \geq 2$, $\varphi_y$ is smooth and 
\[
D^2\varphi_y(x) \leq \frac{q-1}{\eps^{q-1}}|x-y|^{q-2}I.
\]
Choosing 
\[
C =  \frac{q-1}{2\eps^{q-1}} r(\eps)^{q-2} 
\] 
yields $D^2\varphi_y(x) \leq 2CI$ for every $y \in \Omega \cap  B_{r(\eps)}(x)$. This implies that the function 
\[
 \varphi_y(x) - C|x|^2
\]
is concave for every $y \in \Omega  \cap  B_{r(\eps)}(x)$. By taking an infimum over $y \in B_{r(\eps)}(x)$ we conclude that 
\[
 \inf_{y \in \Omega  \cap  B_{r(\eps)}(x)} (\varphi_y(x) - C|x|^2 ) = u_{\eps}(x) - C|x|^2
\]
is concave. 
\end{proof}

\begin{remark}
From (\ref{semiconcave1}) it follows that $u_{\eps}$ is twice differentiable almost everywhere and
\[
D^2u_{\eps}(x) \leq 2C I \qquad \text{a.e.} \,\,x \in \Omega_{r(\eps)}.
\]
\end{remark}

Next we present the proofs for Lemmas \ref{lemma_p<2} and \ref{p<2lemma2}.

\begin{proof}[Proof of Lemma \ref{lemma_p<2}] 
Claim
(i) follows directly from the fact that $Y_{\eps}(x)$ is a closed set. Indeed, fix a point $x_0 \in \Omega$ and let 
$x_k \to x_0$. Choose $y_k \in Y_{\eps}(x_k)$ such that
\[
|y_k - x_k|= \max_{y \in Y_{\eps}(x_k)} |y-x_k|.
\] 
If $y_0$ is any cluster point of $(y_k)$, then $y_0 \in Y_{\eps}(x_0)$. This implies that 
\[
\limsup_{k \to \infty} \max_{y \in Y_{\eps}(x_k)} |y-x_k| = \limsup_{k \to \infty} |y_k - x_k| \leq 
\max_{y \in Y_{\eps}(x_0)} |y-x_0| .
\]  

To prove (ii), suppose that $x$ is such that the gradient $D u_{\eps}(x)$ exists. Fix $y \in  Y_{\eps}(x)$ and let 
$e = \frac{y-x}{|y-x|}$. Then 
\[
 u_{\eps}(x) = u(y) + \frac{|x - y|^q}{q \, \eps^{q-1}} \qquad \text{and} 
\qquad  u_{\eps}(x + he) \leq u(y) + \frac{|x + he - y|^q}{q \, \eps^{q-1}},
\]
which in turn imply
\[
 \begin{split}
 D u_{\eps}(x)\cdot e &= \lim_{h \to 0} \frac{u_{\eps}(x + he) - u_{\eps}(x)}{h} \\
		&\leq  \lim_{h \to 0} \frac{1}{h} \left( u(y) + \frac{|x + he - y|^q}{q \, \eps^{q-1}} -  u(y) -
\frac{|x - y|^q}{q \, \eps^{q-1}}\right) \\
&=  \lim_{h \to 0}  \frac{\left( 1- \frac{h}{|y-x|}\right)^q -1}{h} \, \left( \frac{|x-y|^q}{q \, \eps^{q-1}} \right)\\	
	&= -\left(\frac{|x -y|}{\eps}\right)^{q-1} .
\end{split}
\]
Since $- |D u_{\eps}(x)| \leq  D u_{\eps}(x)\cdot e$, (ii) follows.
\end{proof}

\begin{proof}[Proof of Lemma \ref{p<2lemma2}] 
Suppose $\hat{x}$ is as in the assumption. From Lemma \ref{lemma_p<2}, it follows that $u_\eps(\hat{x}) = u(\hat{x})$ and therefore 
\[
 u(y)+\frac{\abs{\hat{x}-y}^q}{q \eps^{q-1}} \geq u(\hat{x}) \qquad \text{for every} \, \, y\in \Omega .
\]
Denote $\psi(y) = u(\hat x)-\frac{1}{q \eps^{q-1}}\abs{\hat{x}-y}^q$ and notice that $\psi \in C^2(\Omega)$. 
As $q>\frac{p}{p-1}$, an easy calculation yields
\[
 \lim_{r \to 0} \sup_{y \in B_r(\hat{x})}( -\Delta_p \psi(y)) =0.
\]
On the other hand, as $\psi(\hat x)=u(\hat x)$, $\psi(y)\le u(y)$ for $y\in\Omega$ and 
$u$ is a viscosity supersolution to (\ref{poisson}), we have 
$ \lim\limits_{r \to 0} \sup\limits_{y \in B_r(\hat{x})}( -\Delta_p \psi(y)) \ge f(\hat x)$.
Thus it follows that $ f(\hat{x}) \leq 0$. 
\end{proof}

\end{document}